\documentclass[12pt,abstracton]{scrartcl}
\usepackage{amsmath,amsfonts,amscd,amssymb,amsthm}
\usepackage{dsfont} 
\usepackage[english]{babel}
\usepackage{mathrsfs}
\newtheorem{theorem}{Theorem}
\newtheorem{lem-hand}{Hypothesis}
\newtheorem{definition}{Definition}

\newtheorem{corollary}{Corollary}
\newtheorem{lemma}{Lemma}

\newtheorem{example}{Example}
\newtheorem{remark}{Remark}

\newcommand{\nto}{\mbox{$\;\rightarrow_{\hspace*{-0.3cm}{\small n}}\;$~}}

\newcommand{\ntod}{\mbox{$\;{\to}^{^{\hspace*{-0.3cm}{\small d}}}\;$~}}

\title{\LARGE\sffamily\slshape Limit Theorems in Mallows Distance for Processes with Gibssian Dependence}

\author{\large
	L. Cioletti\footnote{Corresponding author.}
	\\[-0.3cm]
	\footnotesize Departamento de Matem\'atica - UnB
	\\[-0.3cm]
	\footnotesize 70910-900, Bras\'ilia, Brazil
	\\[-0.3cm]
	\footnotesize\texttt{cioletti@mat.unb.br}
	\and
	\large
	C. C. Y. Dorea
	\\[-0.3cm]
	\footnotesize Departamento de Matem\'atica - UnB
	\\[-0.3cm]
	\footnotesize 70910-900, Bras\'ilia, Brazil
	\\[-0.3cm]
	\footnotesize\texttt{c.c.y.dorea@mat.unb.br}
	\and
	\large
	R. Vila 
	\\[-0.3cm]
	\footnotesize Departamentos de Matem\'atica e Estat\'istica - UnB
	\\[-0.3cm]
	\footnotesize 70910-900, Bras\'ilia, Brazil
	\\[-0.3cm]
	\footnotesize\texttt{rovig161@mat.unb.br}
}
\date{}

\begin{document}
    \makeatletter
    \def\blfootnote{\gdef\@thefnmark{}\@footnotetext}
    \let\@fnsymbol\@roman
    \makeatother

\maketitle

\begin{abstract}
In this paper, we explore the connection between convergence in 
distribution and Mallows distance in the context of positively associated random variables. 
Our results extend some known invariance principles for sequences with FKG property. 
Applications for processes with Gibbssian dependence structures are included. 
\end{abstract}

\blfootnote{\textup{2010} \textit{Mathematics Subject Classification}: 
	60B10,
	60F05, 
	60G10,
	60K35.}

\section{Introduction}\label{Introduction}

Positive association for a random vector $(X_1,X_2,\cdots,X_n)$ requires that 
\begin{equation}
\mathrm{cov}\big(
g(X_1,\cdots,X_n),		
h(X_1,\cdots,X_n)\big)
\geqslant 0,
\label{associado}
\end{equation}
whenever $g$ and $h$ are two real-valued coordinatewise nondecreasing functions and whenever the covariance exists. 
This dependence structure has been widely used in the studies of reliability theory, see Barlow and Proschan \cite{BP75}. 
The basic concept actually appeared in Harris  \cite{Harris60} in the context of percolation models 
and it was subsequently generalized to a large class of Statistical Mechanics models 
in the seminal work by Fortuin, Kasteleyn and Ginibre \cite{FKG71}; in the Statistical
Mechanics literature this notion was developed independently from reliability
theory, variables are said to satisfy the FKG inequality if they are associated 
(see, e.g., \cite{FKG71,Lebowitz72}). In fact, we say a process
${\pmb X}\equiv \{X_i: i\in\mathbb{Z}\}$ satisfies the FKG property if  
\eqref{associado} holds for any finite subvector $(X_{i_1},X_{i_2},\cdots,X_{i_n})$.
\medskip 

We will make use of the Mallows distance to analyse the asymptotic behavior of positively associated processes. Mallows distance $d_r(F,G)$, also known as Wasserstein or Kantorovich distance, measures the discrepancy betweem two (cumulative) distribution function (d.f.) $F$ and $G$. The upper Fr\`echet bound $H(x,y)=F(x)\wedge G(y)$ illustrates its connection with positive associativity. Let $X\stackrel{d}{=}F$ and $Y\stackrel{d}{=}G$, being $\stackrel{d}{=}$ equality in distribution. Then from the classical Hoeffding's formula we have,
\[
\mathrm{cov}(X,Y)=\int_{\mathbb{R}^2} \big( H(x,y) - F(x)G(y)\big) dxdy.
\]
On the other hand, the representation theorem from Dorea and Ferreira \cite{DF12} allow us to write,
\[
d_r(F,G)= \int_{\mathbb{R}^2} |x-y|^r dH (x,y),\quad\mbox{if}\,\, r\geqslant 1.
\]
Besides an extensive applications to a wide variety of fields, this metric has been successfully used to derive Central Limit Theorem (CLT) type results for heavy-tailed stable distributions (see, e.g., Johnson and Samworth \cite{JS2005} or Dorea and Oliveira \cite{DO14}).  A key property to achieve these results is provided by its close relation to convergence in distribution ($\ntod$), as 
established by Bickel and Freedman \cite{BF81},
\begin{equation}
\vspace*{2mm}
\label{BickelFriedman}
d_r(F_n,G) \nto 0 \; \Leftrightarrow \; F_n \ntod G \;\mbox{and}\; 
\int_{\mathbb{R}}|x|^r dF_n(x) \nto \int_{\mathbb{R}} |x|^r dG (x).
\end{equation}

For stabilized partial sum of positively associated random variables (r.v.'s) 
we will show convergence in Mallows distance and hence the asymptotic normality. 
Theorems \ref{convergencia em Mallows 2} and \ref{teo-conv-dist-mallows-alpha-grande} generalize   
Newman and Wright's \cite{CN81} CLT for stationary processes. 
In a recent preprint \cite{2016arXiv160305322G}, using the Stein's method, 
explicit bounds on the Mallows distance of order $r=1$ is obtained 
under weak stationarity assumptions. Using the same method the authors proved in \cite{EM10} some 
convergence rates in limit theorems of normalized partial sums for certain sequences of dependent, identically distributed
random variables which arise in statistical mechanics.
Under strict stationarity and weakly positive association, 
the authors obtained in \cite{DD88} asymptotic normality and give a bound on the Kolmogorov distance. 
By making use of asymptotic normality we strengthen some of the mentioned results to Mallows $d_r$ convergence. 
As for the non-stationary case, our Theorem \ref{theorem for nonstationary} extends
Cox and Grimmett's \cite{CG84} results. Its proof is conceptually different 
from the Cox and Grimmett's proof  and, in particular, 
we shall mention that the characteristic functions  
does not play a prominent role in our proof. 
\medskip

As application we exhibit the $d_r$ convergence for ferromagnetic Ising type models with
discrete and continuous spins. The results apply to both short and long-range potentials 
and also to non-translation invariant systems. For finite range potentials 
the convergence in the Mallows distance of stabilized sums are obtained for any $r\geq 2$.
To prove similar results for long-range potentials, 
near to the critical temperature seems to be a 
very challenging problem. Here we are able to 
show that the convergence in the Mallows distance still occurs 
but some strong restrictions on the order $r$ 
have to be placed.

\section{Positive Association and Mallows Distance}

Let $\mathbb{Z}$ be the set of integers. We will be considering processes ${\pmb X}\equiv \{X_j: j\in\mathbb{Z}\}$ defined on some probability space 
$(\Omega,\mathscr{F},\mathbb{P})$ and that are positively associated.

\begin{definition}
A process ${\pmb X}$ is said to be positive associated if,
	given two coordinatewise non-decreasing functions 
	$f,g:\mathbb{R}^n\to \mathbb{R}$ 
	and  $j_1,\cdots,j_n\in\mathbb{Z}$, we have
		\[
		\mathrm{cov}\big(
		f(X_{j_1},\ldots, X_{j_n})
		,
		g(X_{j_1},\ldots, X_{j_n})
		\big)
		\geqslant 
		0,
		\]
	provided the covariance exists.
\end{definition}
We say that a function $f:\mathbb{R}^n\to\mathbb{R}$ is non-decreasing if 
$f(x_1,\ldots,x_n)\leqslant f(y_1,\ldots,y_n)$ whenever $x_j\leqslant y_j$ for all $j=1,\ldots, n$. For the sake of notation,
if a different
 probability measure $\mu$ is to be associated with the measurable space $(\Omega,\mathscr{F})$ we shall write 
$\mathrm{cov}_\mu$ and similarly $\mathbb{E}_\mu$ for the expectation. Below we gather few properties needed for our proofs,
see Newman and Wright \cite{CN81} or Oliveira \cite{Oliveira12}.
\begin{lemma}
	\label{lemma-associatedproperties}
Let ${\pmb X}$ be positive associated.
\medskip

(a) For $m_j\ge 1$, if $f_j:\mathbb{R}^{m_j}\to\mathbb{R}$ are coordinatewise non-decreasing functions then
$\{f_j(X_{i_1},\cdots, X_{i_{m_j}}): i_1,\cdots,i_{m_j}\in\mathbb{Z}\}$ is also positive associated.
\medskip

(b) If all $X_j$'s possess finite second moment then the characteristic functions $\displaystyle\phi_j(r_j)=\mathbb{E}(\exp\{i r_jX_j\})$ and 
$\displaystyle\phi(r_1,\cdots,r_n)=\mathbb{E}(\exp\{i\sum_{j=1}^n r_jX_j\})$  satisfy,
\[
\left|\phi(r_1,\cdots,r_n)-\prod_{j=1}^{n}\phi_j(r_j)\right|
\leqslant
{1\over 2}\sum_{1\leqslant j\neq k\leqslant n}|r_jr_k|\mathrm{cov}(X_j,X_k).
\]
\end{lemma}
\begin{definition}(Mallows  \cite{CLM72})
	For $r>0$, the Mallows $r$-distance between d.f.'s
	$F$ and $G$ is given by
\begin{equation}\label{Mallowsdistance}
	d_{r}(F,G)
	=
	\inf_{(X,Y)}
	\big\{\mathbb{E}(|X-Y|^{r})\big\}^{1/{r}},
	\quad
	X\stackrel{d}{=}F,\; Y\stackrel{d}{=}G
\end{equation}
where the infimum is taken over all random vectors $(X,Y)$ with marginal distributions $F$ and $G$, respectively.
\end{definition}

For $r\ge 1$ the Mallows distance represents a metric on the space of d.f.'s and bears a close connection with weak convergence given by  \eqref{BickelFriedman}. Let 
\[
\mathcal{L}_{r} 
=
\big\{
F:\ F\ \text{a d.f.}\,\,, 
\int_{\mathbb{R}} |x|^{r}\,  dF(x)<+\infty
\big\}.
\]
\begin{theorem}(Bickel and Freedman  \cite{BF81})
\label{BickelFriedmantheorem}
Let $r\geqslant 1$ and let the d.f.'s $G$ and $\{F_n\}_{n\geq 1}$ in $\mathcal{L}_{r}$. Then $d_{r}(F_n,F)\nto 0$ if and only if \eqref{BickelFriedman} holds or equivalently, for every bounded continuous function $g:\mathbb{R}\to \mathbb{R}$ we have,
\[
\displaystyle \int_{\mathbb{R}} g(x)\, dF_n(x)
		\nto
		\int_{\mathbb{R}}g(x)\, dG(x) \quad\mbox{and}\quad
\displaystyle\int_{\mathbb{R}}|x|^{r}\, dF_n(x)\nto
\int_{\mathbb{R}}|x|^{r}\, dG(x).
\]
\end{theorem}

Assume $X\stackrel{d}{=}F$, $Y\stackrel{d}{=}G$ and $(X,Y)\stackrel{d}{=}H$, where $H(x,y)=F(x)\wedge G(y)$. Then the following representation result will
be helpful to evaluate $d_r (F,G)$.
\begin{theorem} (Dorea and Ferreira  \cite{DF12}) 
\label{teorema de representacao Mallows}
For $r\geqslant 1$ we have
\begin{eqnarray*}
	d_r^r(F,G)
	=
	\mathbb{E}\big\{|F^{-1}(U)-G^{-1}(U)|^r\big\}
	=
	\int_{0}^{1}
	|F^{-1}(u)-G^{-1}(u)|^r du
	\\[0,2cm]
	=
	\mathbb{E}_{H}\big\{|X-Y|^r\big\}
	= 
	\int_{\mathbb{R}^2} |x-y|^r dH(x,y),
	\end{eqnarray*}
	where $U$ is uniformly distributed on the interval $(0,1)$ and
	\[
	F^{-1}(u)= \inf\{x\in\mathbb{R}:F(x)\geqslant u\},
	\quad
	0 < u < 1,
	\]
	denote the generalized inverse.
\end{theorem}

\section{Asymptotics for Positive Associated and Stationary Sequences}
\label{The Mallows Distance}

Let ${\pmb X}\equiv \{X_j: j\in\mathbb{Z}\}$ be a stationary sequence in the sense that for all $m\geqslant 1$ and $ l\in \mathbb{Z}$,
\[ 
(X_{i_1},\cdots, X_{i_m})\stackrel{d}{=}(X_{i_1+l},\cdots, X_{i_m+l}).
\]
For stochastic process ${\pmb X}$ it is natural, when dealing with limit theorems, to consider blocks of $n$ consecutive variables,
\[
S_n=\sum_{j=1}^n X_j\quad \mbox{and} \quad S_{[k,k+n)}=\sum_{j=k}^{k+n-1} X_j.
\]
Clearly, under stationary assumption we have $S_{[k,k+n)}\stackrel{d}{=}S_n$ for all $k\in\mathbb{Z}$. Our first result follows from  Newman's CLT:

\begin{theorem}(Newman \cite{CN80})
\label{theoremCLT-Newman-FKG} 
Let ${\pmb X}$ be a stationary and positive associated process. Assume that the variance is finite and strictly positive,
$0<\mathrm{var}X_1 <+\infty$, and that 
\begin{equation}
\sigma^2 \equiv  \mathrm{var}(X_1) + 2 \sum_{j\geqslant 2} \mathrm{cov}(X_1,X_j) < +\infty.
\label{variance}
\end{equation}
Then
\begin{eqnarray}\label{eq-clt-newman}
\frac{S_{[k,k+n)}-n\mathbb{E}(X_1) }{\sqrt{n}\sigma}
\ntod
\mathrm{N}(0,1), \quad \forall\,k\in\mathbb{Z}.
\end{eqnarray}
\end{theorem}
It is worth mentioning that the positive associativity and stationarity assures that
\[
\sigma^2 
= 
\chi
\equiv 
\sup_{k\in\mathbb{Z}}
\sum_{j\in \mathbb{Z} } \mathrm{cov}(X_k,X_j)
\]
is well-defined and the latter is known as the susceptibility associated to ${\pmb X}$.
\medskip

Define
\begin{equation}
V_{[k,k+n)}=\frac{S_{[k,k+n)}-n\mathbb{E}(X_1)}{\sqrt{n}\sigma}\stackrel{d}{=}F_{[k,k+n)}
\label{notation V and F}
\end{equation}
and let $\Phi$ be the d.f. of $\mathrm{N}(0,1)$.

\begin{theorem}
\label{convergencia em Mallows 2}
Under the hypotheses of Theorem \ref{theoremCLT-Newman-FKG} we have for $0 < r \leqslant 2$
		\[
		\lim_{n\to\infty} d_r(F_{[k,k+n)},\Phi) = 0.
		\]
\end{theorem}
\begin{proof}
First, note that by stationarity we have
\[
\mathrm{var}(S_{[k,k+n)})=\mathrm{var}(S_n)=
n\mathrm{var}(X_1) + 2(n-1) \sum_{j=2}^n \mathrm{cov}(X_1,X_j).
\]
From \eqref{variance} it follows that $\displaystyle\frac{\mathrm{var}(S_{[k,k+n)})}{n}\nto \sigma^2$. Thus
\begin{equation}
\label{convergencevariance}
\mathbb{E}(V_{[k,k+n)}^2)=\mathbb{E}\big\{\big(\frac{S_{[k,k+n)}-n\mathbb{E}(X_1)}{\sqrt{n}\sigma}\big)^2\big\}\nto 1=\mathbb{E}(Z^2),
\end{equation}
where $Z\stackrel{d}{=}\Phi$. Clearly $V_{[k,k+n)}\in\mathcal{L}_2$. Since the convergence \eqref{eq-clt-newman} holds we conclude from Theorem
\ref{BickelFriedmantheorem} that $d_2(F_{[k,k+n)},\Phi) \nto 0$.
\medskip

Next, to extend the convergence for $0<r<2$ we make use of the representation Theorem \ref{teorema de representacao Mallows}. There
exists a r.v. $Z^*\stackrel{d}{=}\Phi$ such that
the joint distribution of $(V_{[k,k+n)},Z^*)$ is given by $H(x,y)=F_{[k,k+n)}(x)\wedge\Phi(y)$ and
\[
d_2^2(F_{[k,k+n)},\Phi)=\mathbb{E}\big\{(V_{[k,k+n)}-Z^*)^2\big\}\nto 0.
\]
By \eqref{Mallowsdistance} and the Liapounov's inequality we have for $0 < r \leqslant 2$
\[
d_r^r(F_{[k,k+n)},\Phi)\leqslant\mathbb{E}\big\{|V_{[k,k+n)}-Z^*|^r\big\}\nto 0.
\]
\end{proof}

To derive convergence for higher order $d_r$, further moment conditions on $X_j$'s will be required.
For $k\in\mathbb{Z}$ let $u_k(\cdot)$ denote the Cox-Grimmet coefficient defined by
\begin{equation}
\label{CoxGrimmet coefficient}
u_k(n)
=
\sum_{j\in \mathbb{Z}: |k-j|\geqslant n} \mathrm{cov}(X_k,X_j),
\quad n\geqslant 0.
\end{equation}
Since we are assuming stationarity we may take $u(n)=u_k(n)=\sum_{j\in \mathbb{Z}: |j|\geqslant n} \mathrm{cov}(X_0,X_j)$.
Note that, by Lemma \ref{lemma-associatedproperties} the process $\{X_j -\mathbb{E}(X_j): j\in\mathbb{Z}\}$ is also stationary and positive associated.
This allow us to state a moment inequality from Birkel  \cite{TB88} adapted for our needs.
\begin{lemma}\label{lemma Birkel}
Let $2<r<r^*$ and let ${\pmb X}$ be a stationary and positive associated process. Assume that $\mathbb{E}\{|X_1|^{r^*}\} <+\infty$ and that for some constants $C_1 > 0$ and $\theta \geqslant \displaystyle\frac{r^*(r-2)}{2 (r^* -r)}$ we have $u(n) \leqslant C_1 n^{-\theta}$. Then there exist a constant 
$C_2=C_2 (r,r^*)$ such that
\begin{equation}
\label{Birkel bound}
	\sup_{k\in \mathbb{Z}}
	\mathbb{E}\big\{|S_{[k,k+n)} - n\mathbb{E}(X_1)|^r\big\}
	\leqslant
	C_2n^{r/2}.
\end{equation}
\end{lemma}

Note that, under Theorem \ref{convergencia em Mallows 2}, we have the above conditions satisfied for $r=2$. Indeed, by \eqref{variance} we have $u(n)\leqslant C_1$ and \eqref{Birkel bound} follows from \eqref{convergencevariance}.	

\begin{theorem}
\label{teo-conv-dist-mallows-alpha-grande}
Let $2<r<r^*$ and assume that ${\pmb X}$ satisfies the hypotheses of Lemma \ref{lemma Birkel} with $\theta >\displaystyle\frac{r^*(r-2)}
{2 (r^* -r)}$. Then if $\sigma^2$, given by \eqref{variance}, is such that
$ 0< \sigma^2 < +\infty$ we have
\[
d_r(F_{[k,k+n)},\Phi) \nto 0 \quad \mbox{and} \quad \mathbb{E}(|V_{[k,k+n)}|^r)\nto \mathbb{E}(|Z|^r),
\]
where $F_{[k,k+n)}$ and  $V_{[k,k+n)}$ are defined by \eqref{notation V and F} and $Z\stackrel{d}{=}\Phi\stackrel{d}{=}\mathrm{N}(0,1)$.
\end{theorem}

\begin{proof} (i) Since $r^* > 2$, by Theorem \ref{theoremCLT-Newman-FKG} we have $V_{[k,k+n)} \ntod Z$. Next, we show that 
\begin{equation}\label{cvgmoments}
V_{[k,k+n)}\in\mathcal{L}_{r}\quad\mbox{and}\quad \mathbb{E}\{|V_{[k,k+n)}|^r\}\nto \mathbb{E}\{|Z|^r\}.
\end{equation}
Then $d_r(F_{[k,k+n)},\Phi) \nto 0$ follows immediately from \eqref{BickelFriedman}.
\medskip

(ii) For \eqref{cvgmoments} we will show that $\displaystyle\sup_{n\geqslant 1}\sup_{k\in \mathbb{Z}}\mathbb{E}(|V_{[k,k+n)}|^{r'})<+\infty$ for some $r<r'<r^*.$ Thus $\displaystyle V_{[k,k+n)}\in\mathcal{L}_{r'}\subset\mathcal{L}_r$ and the convergence of moments follows from the fact that $\displaystyle\{|V_{[k,k+n)}|^r\}_{n\geqslant 1}$ is uniformly integrable (cf. Billingsley \cite{Billingsley68}, Theorem 5.4).
\medskip

Now let $\psi(r)=\displaystyle\frac{r^*(r-2)}{2 (r^* -r)}$. Then $\psi'(\cdot) >0$ for $r>2$. It follows that there exist $r' > r$ such that $\theta > \psi(r')$. Just take $r' = \displaystyle\frac{2r^*(1+\theta)}{2\theta + r^*}$. From Lemma \ref{lemma Birkel} we have for $C_2 = C_2(r',r^*)> 0$,
 \[
\sup_{k\in \mathbb{Z}}\mathbb{E}\big\{|V_{[k,k+n)}|^{r'}\big\}\leqslant C_2 n^{r'/2}.
\]
It follows that,
 \[
\sup_{k\in \mathbb{Z}}\mathbb{E}\big\{\displaystyle\frac{|V_{[k,k+n)}|^{r'}}{(\sqrt{n}\sigma)^{r'}}\big\}\leqslant C_2 
\displaystyle\frac{n^{r'/2}}{(\sqrt{n}\sigma)^{r'}}=\displaystyle\frac{C_2}{\sigma^{r'}}<+\infty.
\]
\end{proof}

\section{The Non-Stationary Case}
\label{mallows-not-stationary}

When stationarity is relaxed a more refined treatment needs to be carried out. The basic idea is to subdivide the partial sum $S_{[k,k+n)}=\sum_{j=k}^{k+n-1} X_j$ into blocks 
\begin{equation}
\label{blocks}
 S_{[k,k+l_n)},\, S_{[k+l_n,k+2l_n)},\, \ldots,\, S_{[k+(m_n-1)l_n,k+m_nl_n)},\, S_{[k+m_nl_n,k+n)}
\end{equation}
where the first $m_n=\displaystyle[n/l_n]$ 
(the largest integer contained in) blocks
have size $l_n$. 
Note that the last sum in \eqref{blocks}: $S_{[k+m_nl_n,k+n)}$,  have at most 
$n-m_nl_n$ terms, which is non-trivial in case $l_n$ is not a divisor of $n$.
As will be shown, by suitably choosing $l_n$ (see \eqref{choice of ln} ) and by assuming boundness conditions on Cox-Grimmet coefficient \eqref{CoxGrimmet coefficient}, the blocks can be made asymptotically independent. The following arguments suggest the required conditions. Let
\[
\sigma^2_{[k,k+n)} = \mathrm{var}\{S_{[k,k+n)}\}\quad \mbox{and} \quad s_{m_n}^2 = \sum_{j=1}^{m_n}\mathrm{var}\{S_{[k+ (j-1)l_n,k+jl_n)}\}.
\]
From positive associativity we have $\mathrm{cov}(X_r,X_s)\geqslant 0$ and
\begin{align}
\nonumber
\mathrm{cov}\big (S_{[k,k+l_n)},\sum_{1<s\leqslant m_n} S_{[k+(s-1)l_n,k+sl_n)} \big )&=
\sum_{k\leqslant r < k+l_n}\sum_{k+l_n < s \leqslant m_n}\mathrm{cov}\big (X_r,X_s\big )\\
\nonumber
&\leqslant\sum_{k\leqslant r < k+l_n}\sum_{|s-r|\geqslant k+l_n -r}\mathrm{cov}\big (X_r,X_s\big )\\
\nonumber &= \sum_{k\leqslant r < k+l_n} u_r(k+l_n-r) .
\end{align}
The non-stationarity can be bypassed if $u_r(\cdot)$ can be bounded by a stationary sequence $v:\mathbb{Z}^+\to\mathbb{R}$ such that
$u_r(\cdot) \leqslant v(\cdot)$. In this case, the same arguments show that for $ r<s$ 
\[
\mathrm{cov}\big (S_{[k+(r-1)l_n,k+rl_n)},\sum_{r<s\leqslant m_n} S_{[k+(s-1)l_n,k+sl_n)} \big )
\leqslant \sum_{j=1}^{l_n} v(j)
\]
and
\begin{equation}
\label{BoundCovariance}
\sum_{k\leqslant r\neq s <k+l_n} \mathrm{cov}\big (S_{[k+(r-1)l_n,k+rl_n)}, S_{[k+(s-1)l_n,k+sl_n)} \big )
\leqslant 2m_n\sum_{j=1}^{l_n} v(j).
\end{equation}
It follows that,
\begin{equation}\label{bound for variance of sum}
s_{m_n}^2 \leqslant \sigma^2_{[k,k+m_nl_n)}
\leqslant s_{m_n}^2 +2m_n\sum_{j=1}^{l_n} v(j).
\end{equation}
Note that $\displaystyle \sigma^2_{[k,k+m_nl_n)}=\mathrm{var}\big\{\sum_{j=1}^{m_n} S_{[k+(j-1)l_n,k+jl_n)}\big\}$.
Thus, we get ``nearly independence" $\sigma^2_{[k,k+m_nl_n)}
\approx s_{m_n}^2 $ provided the last term can be properly controlled. This leads to:
\medskip
\begin{lem-hand}\label{main hypotheses}	Let ${\pmb X}$ be a positive associated process satisfying :
\vspace*{2mm}

(a) there exists a constant $c>0$ such that $\displaystyle\mathrm{var}\{X_j\} > c$;
\vspace*{2mm}

(b) there exists a function $v:\mathbb{Z}^+\to\mathbb{R}$ such that
\begin{equation}\label{function v}
\sum_{n\geqslant 0} v(n)<\infty\quad\mbox{and}\quad u_j(n)\leqslant v(n),\quad \forall j\in \mathbb{Z}\,,\,\forall n\geqslant 0.
\end{equation}

\end{lem-hand}
\begin{remark}
For the stationary case condition (a) is a necessary assumption, or else, all the variables would be constants. A weaker condition
\[ 
\displaystyle\lim_{n \to \infty}\frac{1}{n}\sum_{j=k}^{k+n}\mathrm{var}\{X_j\} > c, \quad \forall
k\in \mathbb{Z}
\]
could have been assumed.  Also, condition (b) could have been replaced by 
: $\sum_{n\geqslant 0} u_j(n) <\infty$ uniformly on $j\in\mathbb{Z}$.
\end{remark}

\begin{lemma}
\label{few properties}
Assume that Hypothesis \ref{main hypotheses} holds. Then for $k_1 < k_2$,
\begin{equation}
\label{bounds variance}
\displaystyle (k_2-k_1) c \leqslant \sigma^2_{[k_1,k_2)} \leqslant (k_2-k_1) v(0) \quad \mbox{and} \quad m_nl_n c \leqslant s_{m_n}^2 \leqslant
m_nl_nv(0).
\end{equation}
Moreover, if $l_n\nto \infty$ and $\displaystyle \frac{n}{l_n} \nto \infty$ then
\begin{equation}
\label{convergence variance}
\displaystyle \frac{\sigma^2_{[k,k+n)}}{\sigma^2_{[k,k+m_nl_n)}}\nto 1 \quad \mbox{and} \quad 
\frac{\sigma^2_{[k,k+m_nl_n)}}{s_{m_n}^2}\nto 1.
\end{equation}
\end{lemma}
\begin{proof} (i) Note that, from the positivity of the covariances we have
\begin{align}
\nonumber
\displaystyle\sigma^2_{[k_1,k_2)}&=\sum_{k_1\leqslant r,s <k_2} \mathrm{cov}(X_r,X_s)\\
&\leqslant \sum_{r=k_1}^{k_2-1} \sum_{s\in \mathbb{Z} : | r-s| \geqslant 0}\mathrm{cov}(X_r,X_s)
\leqslant (k_2-k_1) v(0).
\nonumber
\end{align}
On the other hand, 
\begin{align}
\nonumber
\displaystyle\sigma^2_{[k_1,k_2)}&=\sum_{r=k_1}^{k_2-1}\mathrm{var}\{X_r\} +\sum_{k_1\leqslant r\neq s <k_2} \mathrm{cov}(X_r,X_s)\\
&\geqslant (k_2-k_1)c.
\nonumber
\end{align}

It follows that for $j=1,\cdots,m_n$ we have,
\[
l_nc\leqslant \sigma^2_{[k+(j-1)l_n, k+jl_n)} \leqslant l_nv(0)
\]
and
\[
m_nl_nc \leqslant s^2_{m_n} = \sum_{j=1}^{m_n}\sigma^2_{[k+(j-1)l_n, k+jl_n)}\leqslant m_nl_nv(0).
\]

(ii) The positive association also assures that 
\[
s^2_{m_n}\leqslant \sigma^2_{[k,k+m_nl_n)}\leqslant \sigma^2_{[k,k+n)}
\]
and
\[
\sigma^2_{[k,k+n)} \leqslant \sigma^2_{[k,k+m_nl_n)}+\sigma^2_{[k+m_nl_n,k+n)}.
\]
From \eqref{bounds variance} we have $\displaystyle\frac{\sigma^2_{[k+m_nl_n,k+n)}}{\sigma^2_{[k,k+m_nl_n)}}\leqslant\frac{(n-m_nl_n)v(0)}{m_nl_nc}$.
Since $\displaystyle \frac {n}{m_nl_n} \nto 1$ we get 
\[
\displaystyle 1\leqslant \frac{\sigma^2_{[k,k+n)}}{\sigma^2_{[k,k+m_nl_n)}}\leqslant 1+\frac{(n-m_nl_n)v(0)}{m_nl_nc}\nto 1.
\]
Similarly, from \eqref{bound for variance of sum} and \eqref{function v} we have
\[
\displaystyle 1\leqslant \frac{\sigma^2_{[k,k+m_nl_n)}}{s_{m_n}^2}\leqslant 1+ 2\frac{m_n\sum_{j=1}^{l_n}v(j)}{m_nl_nc}\nto 1.
\]
\end{proof}

To handle the weak convergence in the non-stationary setup we will make use of the Berry-Esseen
inequality (cf. Feller, vol II, \cite{Feller66} : if $\xi_1,\xi_2,\ldots$ are zero-mean and independent r.v.'s such that 
$\displaystyle\mathbb{E}\{|\xi_j|^3\} <+\infty$ for $j=1,2,\ldots$. Then
\begin{equation}\label{Berry-Esseen}
\sup_{x}\left |\mathbb{P}\left (\frac{\sum_{j=1}^{n}\xi_j}{\sqrt{\mathrm{var}\{\sum_{j=1}^{n}\xi_j\}}}\leqslant x\right )- \Phi(x)\right |\leqslant
6 {\sum_{j=1}^{n}\mathbb{E}(|\xi_j|^3)\over \left(\sum_{j=1}^{n}\mathrm{var}(\xi_j)\right)^{3/2} }.
\end{equation}
This will require a restrictier choice of the block size $l_n$,
\begin{equation}
\label{choice of ln}
l_n\nto\infty,\quad \frac{n}{l_n}\nto\infty\quad \mbox{and}\quad \frac{l_n^3}{m_n}\nto 0.
\end{equation}
Just take, for example, $l_n=n^\delta$ with $\delta < 1/4$.

\begin{theorem} \label{theorem for nonstationary}
Assume ${\pmb X}$ satisfies Hypothesis  \ref{main hypotheses} and that for some constant $C_*$ we have 
$\mathbb{E}\{|X_j|^3\}<C_*<+\infty$ for all $j\in\mathbb{Z}$. Then for $0< r \leqslant 2$ we have
\begin{equation}\label{convergence nonstationary}
\displaystyle d_r \big ( F_{[k,k+n)}, \Phi\big) \nto 0,\quad F_{[k,k+n)}\stackrel{d}{=}\frac{S_{[k,k+n)}-\mathbb{E}\big(S_{[k,k+n)}\big) }{\sigma_
{[k,k+n)}}\quad\mbox{and}\quad\Phi\stackrel{d}{=}\mathrm{N}(0,1).
\end{equation}
\end{theorem}
\begin{proof} (i) Without loss of generality we may assume $\mathbb{E}\{X_j\}=0$ for all $j$. If not, let $X'_j=X_j - \mathbb{E}\{X_j\}$ then the process
$\{X'_j:j\in\mathbb{Z}\}$ satisfies the same hypotheses. Consider the blocks \eqref{blocks} and assume that the block size $l_n$ satisfies \eqref{choice of ln}. We will show that
\begin{equation}\label{cvg blocksum}
\displaystyle d_2 \big ( F_{m_n}, \Phi\big) \nto 0\quad \mbox{with}\quad F_{m_n}\stackrel{d}{=}\frac{S_{[k,k+m_nl_n)}}{s_{m_n}}.
\end{equation}
Assuming \eqref{cvg blocksum} holds then, by Theorem \ref{teorema de representacao Mallows}, there exists $Z^*\stackrel{d}{=}\Phi$ such that
\[
\displaystyle d_2^2 \big ( F_{m_n}, \Phi\big) = \mathbb{E}\big\{\big(\frac{S_{[k,k+m_nl_n)}}{s_{m_n}}- Z^*\big)^2\big\}.
\]
From the definition of Mallows distance \eqref{Mallowsdistance} we have
\[
\displaystyle d_2^2 \big ( F_{[k,k+n)}, \Phi\big) \leqslant \mathbb{E}\big\{\big(\frac{S_{[k,k+n)}}{\sigma_{[k,k+n)}}- Z^*\big)^2\big\}.
\]
Using Minkowski's inequality we have  $d_2 \big ( F_{[k,k+n)}, \Phi\big) \nto 0$ provided
\begin{equation}\label{cvg secondmoment1}
\displaystyle 
A_n=\mathbb{E}\big\{\big(\frac{S_{[k,k+m_nl_n)}}{\sigma_{[k,k+m_nl_n)}}- \frac{S_{[k,k+m_nl_n)}}{s_{m_n}}\big)^2\big\}\nto 0 
\end{equation}
and
\begin{equation}\label{cvg secondmoment2}
\displaystyle
B_n=\mathbb{E}\big\{\big(\frac{S_{[k,k+n)}}{\sigma_{[k,k+n)}}- \frac{S_{[k,k+m_nl_n)}}{\sigma_{[k,k+m_nl_n)}}\big)^2\big\}\nto 0.
\end{equation}
As in the proof of Theorem \ref{convergencia em Mallows 2}, the Liapounov's inequality completes the proof for $0< r <2$.
\medskip

(ii) To show \eqref{cvg secondmoment1} note that $\displaystyle 
\mathbb{E}\big\{\big(\frac{S_{[k,k+m_nl_n)}}{\sigma_{[k,k+m_nl_n)}}\big)^2\big\}=1$. By \eqref{convergence variance} we have
\[
\displaystyle
A_n= \mathbb{E}\big\{\big(\frac{S_{[k,k+m_nl_n)}}{\sigma_{[k,k+m_nl_n)}}\big)^2\big( 1 - \frac{\sigma_{[k,k+m_nl_n)}}{s_{m_n}}\big)^2\big\}
= \big( 1 - \frac{\sigma_{[k,k+m_nl_n)}}{s_{m_n}}\big)^2 \nto 0.
\]
For  \eqref{cvg secondmoment2} write $S_{[k,k+m_nl_n)}=S_{[k,k+n)}-S_{[k+m_nl_n,k+n)}$. Same arguments as above shows that,
\[
\displaystyle
\mathbb{E}\big\{\big(\frac{S_{[k,k+n)}}{\sigma_{[k,k+n)}}-\frac{S_{[k,k+n)}}{\sigma_{[k,k+m_nl_n)}}\big)^2\big\}\nto 0.
\]
Since $\displaystyle\frac{n}{m_nl_n}\nto 1$ we have by \eqref{bounds variance}
\[
\displaystyle
\mathbb{E}\big\{\big(\frac{S_{[k+m_nl_n,k+n)}}{\sigma_{[k,k+m_nl_n)}}\big)^2\big\}\leqslant \frac{(n - m_nl_n)v(0)}{m_nl_n c}\nto 0.
\]
And $B_n\nto 0$.
\medskip

(iii) Since $\mathbb{E}\{|X_j|^3\}<C_*<+\infty$ the results from Lemma \ref{lemma-associatedproperties} can be applied. Taking
$\displaystyle r_1=\cdots=r_{m_n}=\frac{t}{s_{m_n}}$ we get
\[
\displaystyle 
\left|
\mathbb{E}\big (\exp\big \{i\frac{t}{s_{m_n}}\sum_{j=1}^{m_n} S_{[k+(j-1)l_n,k+jl_n)}\big \}\big )
-\prod_{j=1}^{m_n}\mathbb{E}\big (\exp\big \{i\frac{t}{s_{m_n}}S_{[k+(j-1)l_n,k+jl_n)}\big \}\big )
\right|
\leqslant
A(t,k,m_n)
\]
where
\begin{align}
\nonumber
\displaystyle A(t,k,m_n)&= \frac{t^2}{2s_{m_n}^2} \sum_{1 \leqslant r\neq s \leqslant m_n}\mathrm{cov}(S_{[k+(r-1)l_n,k+rl_n)},
S_{[k+(s-1)l_n,k+sl_n)})\\
\nonumber
&= \frac{t^2}{s_{m_n}^2}\sum_{r=1}^{m_n} \mathrm{cov}(S_{[k+(r-1)l_n,k+rl_n)},\sum _{r<s\leqslant m_n}S_{[k+(s-1)l_n,k+sl_n)})\\
\nonumber
&\leqslant \frac{t^2}{s_{m_n}^2} m_n \sum_{j=1}^{l_n}v(j) \leqslant \frac{t^2m_n}{m_nl_nc}\sum_{j=1}^{l_n}v(j) \nto 0.   
\end{align}
For the last inequalities we have used \eqref{BoundCovariance},  \eqref{function v} and \eqref{bounds variance}. 
Now let $Y_j \stackrel{d}{=}S_{[k+(j-1)l_n,k+jl_n)}$ for $j=1,\cdots,m_n$. Assume that $Y_1,\cdots,Y_{m_n}$ are independent r.v.'s. Then we can write.
\begin{equation}\label{characteristicfunction}
\displaystyle  
\left|\mathbb{E}\big (\exp\big\{i\frac{t}{s_{m_n}}\sum_{j=1}^{m_n} S_{[k+(j-1)l_n,k+jl_n)}\big\}\big )
-\prod_{j=1}^{m_n}\mathbb{E}\big (\exp\big\{i\frac{t}{s_{m_n}}Y_j\big\}\big)\right|
\nto 0.
\end{equation}
\medskip

(iv) By \eqref{characteristicfunction} we have for $Z\stackrel{d}{=}\Phi$,
\begin{equation}\label{cvg in distribution}
\displaystyle \frac{\sum_{j=1}^{m_n} Y_j}{s_{m_n}}\ntod Z 
\quad \Longrightarrow \quad
\frac{S_{[k,k+m_nl_n)}}{s_{m_n}}\ntod Z.
\end{equation}
And this will be accomplished using the Berrey-Esseen inequality \eqref{Berry-Esseen}. For every fixed $n$ we have by Minkowski's 
inequality
\[
\displaystyle\big [\mathbb{E}\big\{|S_{[k+(j-1)l_n,k+jl_n)}|^3\big\}\big]^{1/3}\leqslant
\sum_{r=1}^{l_n -1}\big [\mathbb{E}\big\{|X_r|^3\big\}\big]^{1/3} \leqslant l_nC_*^{1/3}.
\]
Thus, we have 
$\mathbb{E}\{|Y_j|^3\}=\mathbb{E}\{|S_{[k+(j-1)l_n,k+jl_n)}|^3\}\leqslant l_n^3C_*< +\infty$ 
for every fixed n.
Since the $Y_j's$ are independent we have $\displaystyle \mathrm{var} \{ \sum_{j=1}^{m_n}Y_j\}=s^2_{m_n}$. From \eqref{Berry-Esseen} we have
\[
\sup_{x}\left|\mathbb{P}\left(\frac{\sum_{j=1}^{m_n}Y_j}{s_{m_n}}\leqslant x\right)- \Phi(x)\right|\leqslant
6 \frac{\sum_{j=1}^{m_n}\mathbb{E}(|Y_j|^3)}{s_{m_n}^3}.
\]
By \eqref{bounds variance} we have $s_{m_n}^3 \geqslant (m_nl_nc)^{3/2}$ and
\[
\sup_{x}\left|\mathbb{P}\left(\frac{\sum_{j=1}^{m_n}(Y_j)}{s_{m_n}}\leqslant x\right)- \Phi(x)\right|\leqslant
6 \frac{m_nl_n^3C_*}{(m_nl_n c)^{3/2}} \nto 0.
\]
In the last limit we have used the fact that  $\displaystyle\frac{l_n^3}{m_n}\nto 0$.
\medskip

(v) To complete the proof of \eqref{cvg blocksum} we make use of Theorem \ref{BickelFriedmantheorem}. Note that $\displaystyle \frac{
S_{[k,k+m_nl_n)}}{s_{m_n}} \in \mathcal{L}_2$ and by \eqref{convergence variance} we have
\[
\displaystyle \mathbb{E}\big\{\big ( \frac{
S_{[k,k+m_nl_n)}}{s_{m_n}} \big)^2    \big\}=\frac {\sigma^2_{k,k+m_nl_n)}}{s_{m_n}^2}\nto 1 =\mathbb{E}\{Z^2\}.
\] 
From  \eqref{cvg in distribution} we also have $\displaystyle\frac{S_{[k,k+m_nl_n)}}{s_{m_n}}\ntod Z$. 
And  \eqref{cvg blocksum} follows.
\end{proof}

\begin{corollary}\label{cor-con-Kol-dis}
Assume $\mathbf{X}$ satisfies Hypothesis 1 and that for 
some constant $C_*$ we have $\mathbb{E}\{|X_j|^3\}<C_*<+\infty$ for all $j\in\mathbb{Z}$.
Then we have
\[
\lim_{n\to\infty} d_{K}(F_{[k,k+n)},\Phi) = 0,
\]
where $d_{K}(F,G)=\sup_{x\in\mathbb{R}}|F(x)-G(x)|$ is the Kolmogorov distance between $F$ and $G$.
\end{corollary}

\begin{proof}
Since a standard normal random variable has probability density bounded by $C=1/\sqrt{2\pi }$ 
it follows from Monge-Kantorovich duality that 
\[
d_{K}(F_{[k,k+n)},\Phi)
\leq 
2\sqrt{C d_{1}(F_{[k,k+n)},\Phi) }.
\]
Therefore the convergence 
follows from Theorem \ref{theorem for nonstationary}, with $r=1$. 
\end{proof}

\section{Applications to Gibbsian Dependent Ensembles}
We will be considering processes ${\pmb X}\equiv\{X_j:j\in\mathbb{Z}\}$ defined on $\Omega = E^{\mathbb{Z}}$
where $E\subset \mathbb{R}$ is a measurable subset. Let $\mathcal{B}(E)$ denote the Borel subsets and let $\lambda$ be a probability
measure on $(E,\mathcal{B}(E))$. On the product space $\Omega$ let $\mathscr{F}$ denote the usual $\sigma$-field. Assume that the variables $X_n$ are projections, that is,
for $\omega=(\cdots,\omega_{-1},\omega_0,\omega_1,\cdots)\in\Omega$ we have $X_n=\omega_n$. The probabilities of the ensembles 
$\omega$ will be derived from a given specifications 
${\pmb \gamma}=\{\gamma_{\Lambda}(A|\omega) : A\in \mathscr{F},\omega\in\Omega,\Lambda \subset
\mathbb{Z}, \Lambda\;\; \mbox{finite}\}$ 
formed by a suitalbe family of probability kernels. 
The kernels $\{\gamma_{\Lambda}(\cdot|\cdot)\}_{\Lambda \subset
\mathbb{Z}}$ are candidates for conditional expectations.  Define 
$\mathscr{F}_{\Lambda}\equiv \sigma(X_i: i\in\Lambda)$ and  similarly $\mathscr{F}_{\mathbb{Z}\setminus\Lambda}$. Gibbs measures $\mathscr{G}({\pmb \gamma})$ are defined to be all the probability measures $\mu$  on $(\Omega,\mathscr{F})$ for which $\mathbb{E}_{\mu}(\mathds{1}_{A}|\mathscr{F}_{\mathbb{Z}\setminus\Lambda})(\omega)=\gamma_{\Lambda}(A|\omega)\,,\mu\mbox{-a.s.}$.
\medskip
   
In the examples that follow the specification ${\pmb \gamma}$  
will be given by exponentially decaying probabilities generated by a 
prescribed Hamiltonian $H$. 
Let $\pmb{J}= \{J_{ij}\geqslant 0: i,j\in\mathbb{Z}\ \}$ be a collection 
of real numbers such that  
\[
J_{ii}=0 \quad \mbox{and} \quad \sup_{i\in\mathbb{Z}} \sum_{j\in\mathbb{Z} } J_{ij} <+\infty.
\]
For each finite  $\Lambda\subset \mathbb{Z}$ and $\omega\in\Omega$ define
\[
H_{\Lambda}(\omega) 
= 
\sum_{i,j\in\Lambda} J_{ij}\ \omega_i\omega_j 
+
\sum_{i\in\Lambda, j\in\mathbb{Z}\setminus\Lambda} J_{ij}\ \omega_i\omega_j. 
\]
\medskip

Under the above setting, if for all finite $\Lambda\subset \mathbb{Z}$ and $\omega\in\Omega$ we have
\[
		Z_{\Lambda}(\omega)
		= 
		\int_{ E^{|\Lambda|} }
		\mathds{1}_{\{\sigma_j=\omega_j: \forall j\in\mathbb{Z}\setminus\Lambda \}}(\sigma)
		\exp(H_{\Lambda}(\sigma))
		\prod _{i\in\Lambda} d\lambda(\sigma_i)<+\infty.
	\]
Then for all $A \in \mathscr{F}$, $\Lambda\subset \mathbb{Z}$ finite  and $\omega\in \Omega$, 
\[
		\gamma_{\Lambda}(A,\omega) 
		= 
		\frac{1}{Z_{\Lambda}(\omega)}
		\int_{ E^{|\Lambda|} }
		\mathds{1}_{A}(\sigma)
		\mathds{1}_{\{\sigma_j=\omega_j: \forall j\in\mathbb{Z}\setminus\Lambda \}}(\sigma)
		\
		\exp(H_{\Lambda}(\sigma))
		\
		\prod _{i\in\Lambda} d\lambda(\sigma_i)
	\]
define a specification. 
For further details on this matter see Georgii \cite{georgii}.
\medskip

We are interested on the pair $\lambda$ and $\pmb{J}$ 
that ensure the existence of at least one Gibbs measure 
$\mu\in\mathscr{G}({\pmb \gamma})$. And  such that   
the CLT holds for  
${\pmb X}$ on the probability space $(\Omega,\mathscr{F},\mu)$. More specifically,
for $S_{[k,k+n)}=\sum_{i=k}^{k+n-1}X_i$
\begin{equation}
\label{CLT mu}
\left|
\mu\left(
{
S_{[k,k+n)} - \mathbb{E}_{\mu}(S_{[k,k+n)})
\over 
\sqrt{\mathrm{var}_{\mu}(S_{[k,k+n)}) }
} 
\leqslant x
\right) 
- 
\Phi(x)
\right|
\nto
0.
\end{equation} 
And a stronger result, namely, the convergence in the Mallows distance  
\begin{equation}
\label{Mallows mu}
\lim_{n\to\infty} d_r(F^{\mu}_{S_{[k,k+n)}},\Phi)=0\quad\mbox{with}\,\,F^{\mu}_{S_{[k,k+n)}}\stackrel{\mathrm{d}}{=}{ S_{[k,k+n)} - \mathbb{E}_{\mu}(S_{[k,k+n)})
\over 
\sqrt{\mathrm{var}_{\mu}(S_{[k,k+n)}) }}.
\end{equation}

\begin{example}
Suppose that $E=\mathbb{R}$ and let $\lambda$ be a non-degenerated probability measure on $(\mathbb{R},\mathcal{B}(\mathbb{R}))$ such that $\displaystyle\int_{\mathbb{R}}x^2 d\lambda (x) <+\infty$. Assume that  
$J_{ij}= 0$ for all $i$ and $j\in\mathbb{Z}$. Then $Z_{\Lambda}(\omega)=1$ and $\gamma_{\Lambda}(\cdot,\cdot)$ are well-defined.  
The set 
of the Gibbs measures $\mathscr{G}({\pmb \gamma})$
is a singleton and its unique probability measure $\mu$ is the 
product measure $\mu
=
\prod_{i\in\mathbb{Z}} \lambda_i$ where $\lambda_i = \lambda$, $\forall i\in\mathbb{Z}$.
This is easily verified by noting that for $B\in\mathscr{F}_{\mathbb{Z}\setminus\Lambda}$ and $A\in\mathscr{F}$ we have  
\begin{eqnarray}
\int_B \mathds{1}_{A}(\omega) d\mu (\omega)
&=&
\int_B
\Big[ \int_{\mathbb{R}^{|\Lambda|}}
\mathds{1}_{A}(\sigma)
\mathds{1}_{\{\sigma_j=\omega_j : j\in \mathbb{Z}\setminus\Lambda\}}(\sigma)   
\prod_{i\in\Lambda}d\lambda_i(\sigma_i)\Big]
\prod_{j\in\mathbb{Z}\setminus\Lambda} d\lambda_j(\omega_j)
\nonumber\\
\vspace*{2mm}
&=&
\int_B \gamma_\Lambda (A,\omega)\prod_{j\in\mathbb{Z}\setminus\Lambda} d\lambda_j(\omega_j)=\int_B \gamma_\Lambda (A,\omega)
d\mu(\omega).
\nonumber
\end{eqnarray}
Thus $\mathbb{E}_{\mu}(\mathds{1}_{A}|\mathscr{F}_{\mathbb{Z}\setminus\Lambda})(\omega)=\gamma_{\Lambda}(A|\omega)$. It follows that  we have
a sequence 
of i.i.d r.v.'s with $X_i\stackrel{\mathrm{d}}{=}\lambda$. Since $\lambda$ is non-degenerated we have $\mathrm{var} \{X_i\} > 0$. Clearly the hypotheses of Theorem \ref{convergencia em Mallows 2} are satisfied and the desired convergences 
\eqref{CLT mu} and \eqref{Mallows mu} follow for $r\leq 2$. Moreover, if for some $r^* > 2$ we have $\displaystyle\int_{\mathbb{R}}|x|^{r^*} d\lambda (x) <+\infty$
then by Theorem \ref{teo-conv-dist-mallows-alpha-grande} we also have \eqref{Mallows mu} for $2< r <r^*$.\\
\end{example}
\begin{example}\label{exemplo-ising-curto-alcance}
Let $E=[-1,1]$ and let $\lambda$ be the normalized Lebesgue measure
on $E$. For fixed $L> 0$ define $\pmb J$ by : 
\[
J_{ij}=J\mathds{1}_{\{0<|i-j|\leqslant L\}},\quad i,j\in\mathbb{Z}
\]
where $J>0$ is a constant.
In this case, it is well-known that 
set of the Gibbs measures $\mathscr{G}({\pmb \gamma})=\{\mu\}$
is a singleton. 
A straightforward application of the GKS-II inequality 
shows that $\pmb{X}$ on $(\Omega,\mathscr{F},\mu)$ is not a sequence of
independent r.v.'s. Making use of the FKG inequality one can verify that $\pmb{X}$ on $(\Omega,\mathscr{F},\mu)$ is stationary and positive 
associated. The Lieb-Simon inequality (cf. \cite{Lieb80} and \cite{Simon80}) shows that 
the susceptibility $\chi(\mu)<+\infty$. Thus \eqref{variance} holds and we have from Theorem \ref{convergencia em Mallows 2} the convergences 
\eqref{CLT mu} and \eqref{Mallows mu} for $0< r\leqslant 2$. Moreover, the Lieb-Simon inequality also assures that 
for any $i\in\mathbb{Z}$ we have 
$\mathrm{cov}_{\mu}(X_0,X_i)\leq C e^{-m|i|}$, where
$C$ and $m$ are positive constants. It follows that the hypotheses of Theorem  \ref{teo-conv-dist-mallows-alpha-grande}
are verified and the convergence \eqref{Mallows mu} also holds for $r\geqslant 2$.\\
\end{example}

\begin{example}\label{exemplo-ising-longo-alcance}
Let $E=\{-1,1\}$, $\lambda$ the normalized counting measure
on $E$. For all $i\in\mathbb{Z}$ we define 
$J_{ii}=0$ and 
\[
J_{ij} = \beta |i-j|^{-\alpha}, \quad i,j\in\mathbb{Z}\ \text{and}\ i\neq j
\]
where $\beta>0$ and $\alpha>1$.

%\bigskip
In this example the discussion is much more subtle.
We have to split the analysis in 
terms of the parameter $\alpha$ in two cases\footnote{For $\alpha<1$ the collection $\pmb J$ 
is not even $\lambda$-admissible so this case is 
in some sense is considered trivial.}. 
The first one (and more complex) is $1<\alpha\leqslant 2$ and the 
second is $\alpha>2$. 
To make it explicit the dependence on the parameters $\beta$ and 
$\alpha$, we write $\mathscr{G}({\pmb \gamma}_{\beta,\alpha})$ instead of $\mathscr{G}({\pmb \gamma})$.

Suppose that $1<\alpha\leqslant 2$. 
In this case there is a real number $\beta_c(\alpha)$, 
called critical point satisfying 
$0<\beta_c(\alpha)<+\infty$, such that 
the set of the Gibbs measures $\mathscr{G}({\pmb \gamma}_{\beta,\alpha})$ 
has infinitely many elements for all $\beta>\beta_c(\alpha)$
(supercritical phase)
and for all $\beta<\beta_c(\alpha)$ is a singleton (subcritical phase) 
\cite{FD69, FS82}. 
In the subcritical phase, the unique probability measure $\mu_{\beta,\alpha}$ compatible with 
the specification $\pmb\gamma_{\beta,\alpha}$ has the FKG property and the 
stochastic process ${\pmb X}= \{X_j:j\in\mathbb{Z}\}$ on 
$(\Omega,\mathscr{F},\mu_{\beta,\alpha})$ is associated
and stationary. Aizenman and Newman \cite{AN86} obtained polynomial
decay for $\mathrm{cov}_{\mu_{\beta,\alpha}}(X_0,X_i)$, up to the critical point , i.e., 
the existence of some positive constant $C(\beta,\alpha)$
so that for all $\beta<\beta_c(\alpha)$ 
we have 
$\mathrm{cov}_{\mu_{\beta,\alpha}}(X_0,X_i)\leq C(\beta,\alpha) |i|^{-\alpha}$
and therefore (since $\alpha>1$) 
the susceptibility $\chi(\mu_{\beta,\alpha})<+\infty$ and 
the Cox-Grimmett coefficient satisfies 
$u_{{\pmb X}}(n) = O( n^{1-\alpha})$. 
In this case the convergence \eqref{Mallows mu} holds for
$r=2$ or $r>2$ and $r^2+(\delta-2)r<2\delta\alpha$ for some $\delta>0$ (Lemma \ref{lemma Birkel}).

For $1<\alpha\leqslant 2$ and $\beta_c(\alpha)<\beta$
the analysis is much harder. For example, we can not
ensure that the stochastic process ${\pmb X}$ on 
$(\Omega,\mathscr{F},\mu)$ is stationary
for any $\mu\in\mathscr{G}({\pmb \gamma}_{\beta,\alpha})$. Moreover
the susceptibility is not finite anymore.

The case $\alpha>2$ is similar to the case $1<\alpha\leq 2$
and $\beta<\beta_{c}(\alpha)$, but no restriction on the parameter $\beta$
is need to ensure the uniqueness of the Gibbs measures and the other
used properties.

\end{example}

On each of the three previous examples, the stochastic process   
${\pmb X}$ on $(\Omega,\mathscr{F},\mu)$ is 
stationary. We now present a new example where 
the stationarity hypothesis is broken and the more 
general results of the Section \ref{mallows-not-stationary}
is required to ensure the convergence \eqref{Mallows mu}. 
We remark that for the next example the CLT theorem obtained
by Newman in \cite{CN80} can not used.

\begin{example}\label{exemplo-ising-longo-alcance-nao-invariante}
We take $E=\{-1,1\}$, $\lambda$ the normalized counting measure
on $E$. For all $i\in\mathbb{Z}$ we define 
$J_{ii}=0$ and \[
J_{ij} = |i-j|^{-\alpha}+r_{ij}, \quad i,j\in\mathbb{Z}\ \text{and}\ i\neq j
\]
where $\alpha> 2$ and  $r_{ij}$ is arbitrarily chosen, but satisfying 
for some positive constants $C_1<1$ and $C_2>1$
the following inequalities  
$C_1|i-j|^{-\alpha}\leqslant r_{ij}\leqslant C_2|i-j|^{-\alpha}$.
The family $\pmb J$ is $\lambda$-admissible and the set of the Gibbs measures
$\mathscr{G}({\pmb \gamma})$ still is a singleton. This unique Gibbs measure
$\mu$ has the FKG property and ${\pmb X}$ on $(\Omega,\mathscr{F},\mu)$ 
is not stationary, in general. 

Note that the Hipothesis \ref{main hypotheses}
of the Section \ref{mallows-not-stationary} follows from the 
GKS-II \cite{Griffiths67,KellySherman69} and 
Simon-Lieb \cite{Lieb80} inequalities. 
Since the coordinates of $\pmb{X}$ are uniformly bounded, 
we can apply Theorem \ref{theorem for nonstationary}
to obtain the convergence \eqref{Mallows mu}.
\end{example}

\section{Concluding Remarks}

In \cite{Bu96} there are results similar to Theorem \ref{theorem for nonstationary}
and Corollary \ref{cor-con-Kol-dis}. Although in \cite{Bu96} 
the processes can be indexed in $\mathbb{Z}^d$ the sequence is required to 
have finite $(3+\varepsilon)$ moment, while here only the third moment is required.
In \cite{Bu96} the asymptotic normality of stabilized partial sums 
are proved in the Kolmogorov distance. This result, 
for one-dimensional case, is strengthened (Corollary \ref{cor-con-Kol-dis} )
by proving the convergence in the Mallows distance of order one.   

The results of Section \ref{The Mallows Distance} can be easily 
generalized to multidimensional indexed processes 
since their basic ideas are based on \cite{CG84} and their 
results are valid for multidimensional indexed processes.

\section*{Acknowledgements}
The authors are partially supported by CNPq.

\bibliographystyle{alpha}

%\begin{thebibliography}{vdBK85}

\end{document}